\documentclass[reqno, 12pt]{amsart}
\pdfoutput=1

\def\ssign{\textsection\nobreak\hspace{1pt plus 0.3pt}}
\makeatletter
\let\origsection=\section 
\def\mysection{\@mystartsection{section}{1}\z@{.7\linespacing\@plus\linespacing}{.5\linespacing}{\normalfont\scshape\centering\ssign}}
\def\section{\@ifstar{\origsection*}{\mysection}}
\def\appendix{\par\c@section\z@ \c@subsection\z@
	\let\sectionname\appendixname
	\let\section=\origsection
	\def\thesection{\@Alph\c@section}} 
\def\@mystartsection#1#2#3#4#5#6{\if@noskipsec \leavevmode \fi
	\par \@tempskipa #4\relax
	\@afterindenttrue
	\ifdim \@tempskipa <\z@ \@tempskipa -\@tempskipa \@afterindentfalse\fi
	\if@nobreak \everypar{}\else
	\addpenalty\@secpenalty\addvspace\@tempskipa\fi
	\@dblarg{\@mysect{#1}{#2}{#3}{#4}{#5}{#6}}}
\def\@mysect#1#2#3#4#5#6[#7]#8{\edef\@toclevel{\ifnum#2=\@m 0\else\number#2\fi}\ifnum #2>\c@secnumdepth \let\@secnumber\@empty
	\else \@xp\let\@xp\@secnumber\csname the#1\endcsname\fi
	\@tempskipa #5\relax
	\ifnum #2>\c@secnumdepth
	\let\@svsec\@empty
	\else
	\refstepcounter{#1}\edef\@secnumpunct{\ifdim\@tempskipa>\z@ \@ifnotempty{#8}{\@nx\enspace}\else
		\@ifempty{#8}{.}{\@nx\enspace}\fi
	}\@ifempty{#8}{\ifnum #2=\tw@ \def\@secnumfont{\bfseries}\fi}{}\protected@edef\@svsec{\ifnum#2<\@m
		\@ifundefined{#1name}{}{\ignorespaces\csname #1name\endcsname\space
		}\fi
		\@seccntformat{#1}}\fi
	\ifdim \@tempskipa>\z@ \begingroup #6\relax
	\@hangfrom{\hskip #3\relax\@svsec}{\interlinepenalty\@M #8\par}\endgroup
	\ifnum#2>\@m \else \@tocwrite{#1}{#8}\fi
	\else
	\def\@svsechd{#6\hskip #3\@svsec
		\@ifnotempty{#8}{\ignorespaces#8\unskip
			\@addpunct.}\ifnum#2>\@m \else \@tocwrite{#1}{#8}\fi
	}\fi
	\global\@nobreaktrue
	\@xsect{#5}}
\makeatother

\usepackage{amsmath,amssymb,amsthm, nccmath}
\usepackage{mathrsfs}
\usepackage{mathabx}\changenotsign
\usepackage{dsfont}
\usepackage{bbm}
\usepackage{todonotes}

\usepackage{xcolor}
\usepackage[backref=section]{hyperref}
\usepackage{cleveref}
\usepackage[ocgcolorlinks]{ocgx2}
\hypersetup{
	colorlinks=true,
	linkcolor={red!60!black},
	citecolor={green!60!black},
	urlcolor={blue!60!black},
}

\usepackage[open,openlevel=2,atend]{bookmark}

\usepackage[abbrev,msc-links,backrefs]{amsrefs}
\usepackage{doi}

\renewcommand{\PrintDOI}[1]{\doi{#1}}

\usepackage[T1]{fontenc}
\usepackage{lmodern}
\usepackage[babel]{microtype}
\usepackage[english]{babel}

\usepackage{geometry}
\numberwithin{equation}{section}

\usepackage{enumitem}

\let\polishlcross=\l
\def\l{\ifmmode\ell\else\polishlcross\fi}

\let\emptyset=\varnothing
\let\setminus=\smallsetminus

\makeatletter
\def\moverlay{\mathpalette\mov@rlay}
\def\mov@rlay#1#2{\leavevmode\vtop{   \baselineskip\z@skip \lineskiplimit-\maxdimen
		\ialign{\hfil$\m@th#1##$\hfil\cr#2\crcr}}}
\newcommand{\charfusion}[3][\mathord]{
	#1{\ifx#1\mathop\vphantom{#2}\fi
		\mathpalette\mov@rlay{#2\cr#3}
	}
	\ifx#1\mathop\expandafter\displaylimits\fi}
\makeatother

\newcommand{\bigdcup}{\charfusion[\mathop]{\bigcup}{\cdot}}

\DeclareFontFamily{U}  {MnSymbolC}{}
\DeclareSymbolFont{MnSyC}         {U}  {MnSymbolC}{m}{n}
\DeclareFontShape{U}{MnSymbolC}{m}{n}{
	<-6>  MnSymbolC5
	<6-7>  MnSymbolC6
	<7-8>  MnSymbolC7
	<8-9>  MnSymbolC8
	<9-10> MnSymbolC9
	<10-12> MnSymbolC10
	<12->   MnSymbolC12}{}
\DeclareMathSymbol{\powerset}{\mathord}{MnSyC}{180}

\usepackage{tikz}
\usetikzlibrary{calc,decorations.pathmorphing}
\usetikzlibrary{arrows,decorations.pathreplacing}
\pgfdeclarelayer{background}
\pgfdeclarelayer{foreground}
\pgfdeclarelayer{front}
\pgfsetlayers{background,main,foreground,front}

\usepackage{multicol}
\usepackage{subcaption}

\newcommand{\pedge}[9]{
	
	\ifx\relax#6\relax
	\def\qoffs{0pt}
	\else
	\def\qoffs{#6}
	\fi
	
	\def\phedge{
		($#1+#5!\qoffs!-90:#2-#5$) -- 
		($#2+#1!\qoffs!-90:#3-#1$) -- 
		($#3+#2!\qoffs!-90:#4-#2$) -- 
		($#4+#3!\qoffs!-90:#5-#3$) -- 
		($#5+#4!\qoffs!-90:#1-#4$) -- cycle}
	
	\coordinate (12) at ($#1!\qoffs!90:#2$);
	\coordinate (15) at ($#1!\qoffs!-90:#5$);
	\coordinate (23) at ($#2!\qoffs!90:#3$);
	\coordinate (21) at ($#2!\qoffs!-90:#1$);
	\coordinate (34) at ($#3!\qoffs!90:#4$);
	\coordinate (32) at ($#3!\qoffs!-90:#2$);
	\coordinate (45) at ($#4!\qoffs!90:#5$);
	\coordinate (43) at ($#4!\qoffs!-90:#3$);
	\coordinate (51) at ($#5!\qoffs!90:#1$);
	\coordinate (54) at ($#5!\qoffs!-90:#4$);
	
	\def\nphedge{
		(15) let \p1=($(15)-#1$), \p2=($(12)-#1$) in 
		arc[start angle={atan2(\y1,\x1)}, delta angle={atan2(\y2,\x2)-atan2(\y1,\x1)-360*(atan2(\y2,\x2)-atan2(\y1,\x1)>0)}, x radius=\qoffs, y radius=\qoffs] --
		(21) let \p1=($(21)-#2$), \p2=($(23)-#2$) in 
		arc[start angle={atan2(\y1,\x1)}, delta angle={atan2(\y2,\x2)-atan2(\y1,\x1)-360*(atan2(\y2,\x2)-atan2(\y1,\x1)>0)}, x radius=\qoffs, y radius=\qoffs] --
		(32) let \p1=($(32)-#3$), \p2=($(34)-#3$) in 
		arc[start angle={atan2(\y1,\x1)}, delta angle={atan2(\y2,\x2)-atan2(\y1,\x1)-360*(atan2(\y2,\x2)-atan2(\y1,\x1)>0)}, x radius=\qoffs, y radius=\qoffs] --
		(43) let \p1=($(43)-#4$), \p2=($(45)-#4$) in 
		arc[start angle={atan2(\y1,\x1)}, delta angle={atan2(\y2,\x2)-atan2(\y1,\x1)-360*(atan2(\y2,\x2)-atan2(\y1,\x1)>0)}, x radius=\qoffs, y radius=\qoffs] --
		(54) let \p1=($(54)-#5$), \p2=($(51)-#5$) in 
		arc[start angle={atan2(\y1,\x1)}, delta angle={atan2(\y2,\x2)-atan2(\y1,\x1)-360*(atan2(\y2,\x2)-atan2(\y1,\x1)>0)}, x radius=\qoffs, y radius=\qoffs] --
		cycle}
	
	\ifx\relax#7\relax
	\def\plwidth{1pt}
	\else
	\def\plwidth{#7}
	\fi
	
	\ifx\relax#9\relax
	\fill \nphedge;
	\else
	\fill[#9]\nphedge;
	\fi
	
	\ifx\relax#8\relax
	\draw[line width=\plwidth,rounded corners=\qoffs]\nphedge;
	\else
	\draw[line width=\plwidth,#8]\nphedge;
	\fi
}

\newcommand{\qedge}[7]{
	
	\ifx\relax#4\relax
	\def\qoffs{0pt}
	\else
	\def\qoffs{#4}
	\fi
	
	\def\qhedge{
		($#1+#3!\qoffs!-90:#2-#3$) --
		($#2+#1!\qoffs!-90:#3-#1$) --
		($#3+#2!\qoffs!-90:#1-#2$) -- cycle}

	\coordinate (12) at ($#1!\qoffs!90:#2$);
	\coordinate (13) at ($#1!\qoffs!-90:#3$);
	\coordinate (23) at ($#2!\qoffs!90:#3$);
	\coordinate (21) at ($#2!\qoffs!-90:#1$);
	\coordinate (31) at ($#3!\qoffs!90:#1$);
	\coordinate (32) at ($#3!\qoffs!-90:#2$);
	
	\def\nqhedge{
		(13) let \p1=($(13)-#1$), \p2=($(12)-#1$) in
		arc[start angle={atan2(\y1,\x1)}, delta angle={atan2(\y2,\x2)-atan2(\y1,\x1)-360*(atan2(\y2,\x2)-atan2(\y1,\x1)>0)}, x radius=\qoffs, y radius=\qoffs] --
		(21) let \p1=($(21)-#2$), \p2=($(23)-#2$) in
		arc[start angle={atan2(\y1,\x1)}, delta angle={atan2(\y2,\x2)-atan2(\y1,\x1)-360*(atan2(\y2,\x2)-atan2(\y1,\x1)>0)}, x radius=\qoffs, y radius=\qoffs] --
		(32) let \p1=($(32)-#3$), \p2=($(31)-#3$) in
		arc[start angle={atan2(\y1,\x1)}, delta angle={atan2(\y2,\x2)-atan2(\y1,\x1)-360*(atan2(\y2,\x2)-atan2(\y1,\x1)>0)}, x radius=\qoffs, y radius=\qoffs] --
		cycle}
	
	\ifx\relax#5\relax
	\def\qlwidth{1pt}
	\else
	\def\qlwidth{#5}
	\fi
	
	\ifx\relax#7\relax
	\fill \nqhedge;
	\else
	\fill[#7]\nqhedge;
	\fi
	
	\ifx\relax#6\relax
	\draw[line width=\qlwidth,rounded corners=\qoffs]\nqhedge;
	\else
	\draw[line width=\qlwidth,#6]\nqhedge;
	\fi
}

\newcommand{\redge}[8]{
	
	\ifx\relax#5\relax
	\def\qoffs{0pt}
	\else
	\def\qoffs{#5}
	\fi
	
	\def\rhedge{
		($#1+#4!\qoffs!-90:#2-#4$) -- 
		($#2+#1!\qoffs!-90:#3-#1$) -- 
		($#3+#2!\qoffs!-90:#4-#2$) -- 
		($#4+#3!\qoffs!-90:#1-#3$) -- cycle}
	
	\coordinate (12) at ($#1!\qoffs!90:#2$);
	\coordinate (14) at ($#1!\qoffs!-90:#4$);
	\coordinate (23) at ($#2!\qoffs!90:#3$);
	\coordinate (21) at ($#2!\qoffs!-90:#1$);
	\coordinate (34) at ($#3!\qoffs!90:#4$);
	\coordinate (32) at ($#3!\qoffs!-90:#2$);
	\coordinate (41) at ($#4!\qoffs!90:#1$);
	\coordinate (43) at ($#4!\qoffs!-90:#3$);
	
	\def\nrhedge{
		(14) let \p1=($(14)-#1$), \p2=($(12)-#1$) in 
		arc[start angle={atan2(\y1,\x1)}, delta angle={atan2(\y2,\x2)-atan2(\y1,\x1)-360*(atan2(\y2,\x2)-atan2(\y1,\x1)>0)}, x radius=\qoffs, y radius=\qoffs] --
		(21) let \p1=($(21)-#2$), \p2=($(23)-#2$) in 
		arc[start angle={atan2(\y1,\x1)}, delta angle={atan2(\y2,\x2)-atan2(\y1,\x1)-360*(atan2(\y2,\x2)-atan2(\y1,\x1)>0)}, x radius=\qoffs, y radius=\qoffs] --
		(32) let \p1=($(32)-#3$), \p2=($(34)-#3$) in 
		arc[start angle={atan2(\y1,\x1)}, delta angle={atan2(\y2,\x2)-atan2(\y1,\x1)-360*(atan2(\y2,\x2)-atan2(\y1,\x1)>0)}, x radius=\qoffs, y radius=\qoffs] --
		(43) let \p1=($(43)-#4$), \p2=($(41)-#4$) in 
		arc[start angle={atan2(\y1,\x1)}, delta angle={atan2(\y2,\x2)-atan2(\y1,\x1)-360*(atan2(\y2,\x2)-atan2(\y1,\x1)>0)}, x radius=\qoffs, y radius=\qoffs] --
		cycle}
	
	\ifx\relax#6\relax
	\def\rlwidth{1pt}
	\else
	\def\rlwidth{#6}
	\fi
	
	\ifx\relax#8\relax
	\fill \nrhedge;
	\else
	\fill[#8]\nrhedge;
	\fi
	
	\ifx\relax#7\relax
	\draw[line width=\rlwidth,rounded corners=\qoffs]\nrhedge;
	\else
	\draw[line width=\rlwidth,#7]\nrhedge;
	\fi
}

\let\phi=\varphi
\let\epsilon=\varepsilon
\let\eps=\epsilon
\let\rho=\varrho
\let\theta=\vartheta

\def\NN{{\mathds N}}

\def\RR{{\mathds R}}

\newtheoremstyle{note}  {4pt}  {4pt}  {\sl}  {}  {\bfseries}  {.}  {.5em}          {}
\newtheoremstyle{introthms}  {3pt}  {3pt}  {\itshape}  {}  {\bfseries}  {.}  {.5em}          {\thmnote{#3}}
\newtheoremstyle{remark}  {2pt}  {2pt}  {\rm}  {}  {\bfseries}  {.}  {.3em}          {}

\theoremstyle{plain}
\newtheorem{theorem}{Theorem}
\newtheorem{lemma}[theorem]{Lemma}
\newtheorem{prop}[theorem]{Proposition}

\newtheorem{cor}[theorem]{Corollary}

\newtheorem{obs}[theorem]{Observation}

\theoremstyle{note}

\newtheorem{definition}[theorem]{Definition}

\theoremstyle{remark}

\newtheorem{question}[theorem]{Question}

\usepackage{lineno}
\newcommand*\patchAmsMathEnvironmentForLineno[1]{
	\expandafter\let\csname old#1\expandafter\endcsname\csname #1\endcsname
	\expandafter\let\csname oldend#1\expandafter\endcsname\csname end#1\endcsname
	\renewenvironment{#1}
	{\linenomath\csname old#1\endcsname}
	{\csname oldend#1\endcsname\endlinenomath}}
\newcommand*\patchBothAmsMathEnvironmentsForLineno[1]{
	\patchAmsMathEnvironmentForLineno{#1}
	\patchAmsMathEnvironmentForLineno{#1*}}
\AtBeginDocument{
	\patchBothAmsMathEnvironmentsForLineno{equation}
	\patchBothAmsMathEnvironmentsForLineno{align}
	\patchBothAmsMathEnvironmentsForLineno{flalign}
	\patchBothAmsMathEnvironmentsForLineno{alignat}
	\patchBothAmsMathEnvironmentsForLineno{gather}
	\patchBothAmsMathEnvironmentsForLineno{multline}
}

\def\ex{\text{\rm ex}}

\usepackage{scalerel}

\newsavebox\vdegbox
\savebox\vdegbox{\tikz{
		\draw[black,fill=black] (90:1) circle (.35);
		\draw[black,line width=0.10cm] (210:1) circle (.30);
		\draw[black,line width=0.10cm] (330:1) circle (.30);
		\draw[opacity=0] (0:1.2) circle (0.1);
}}

\newsavebox\pdegbox
\savebox\pdegbox{\tikz{
		\draw[black,line width=0.10cm] (90:1) circle (.30);
		\draw[black,fill=black] (210:1) circle (.35);
		\draw[black,fill=black] (330:1) circle (.35);
		\draw[black,line width=0.28cm ] (210:1) -- (330:1);
		\draw[opacity=0] (0:1.2) circle (0.1);
}}

\newsavebox\vvvbox
\savebox\vvvbox{\tikz{
		\draw[black,fill=black] (90:1) circle (.35);
		\draw[black,fill=black] (210:1) circle (.35);
		\draw[black,fill=black] (330:1) circle (.35);
		\draw[opacity=0] (0:1.2) circle (0.1);
}}

\newsavebox\evbox
\savebox\evbox{\tikz{
		\draw[black,fill=black] (90:1) circle (.35);
		\draw[black,fill=black] (210:1) circle (.35);
		\draw[black,fill=black] (330:1) circle (.35);
		\draw[black,line width=0.28cm ] (210:1) -- (330:1);
		\draw[opacity=0] (0:1.2) circle (0.1);
}}

\newsavebox\eebox
\savebox\eebox{\tikz{
		\draw[black,fill=black] (90:1) circle (.35);
		\draw[black,fill=black] (210:1) circle (.35);
		\draw[black,fill=black] (330:1) circle (.35);
		\draw[black,line width=0.28cm ] (90:1) -- (330:1);
		\draw[black,line width=0.28cm ] (90:1) -- (210:1);
		\draw[opacity=0] (0:1.2) circle (0.1);
}}

\newsavebox\eeebox
\savebox\eeebox{\tikz{
		\draw[black,fill=black] (90:1) circle (.35);
		\draw[black,fill=black] (210:1) circle (.35);
		\draw[black,fill=black] (330:1) circle (.35);
		\draw[black,line width=0.28cm ] (90:1) -- (330:1);
		\draw[black,line width=0.28cm ] (90:1) -- (210:1);
		\draw[black,line width=0.28cm ] (210:1) -- (330:1);
		\draw[opacity=0] (0:1.2) circle (0.1);
}}

\makeatletter
\newcommand{\overrighharpoonup}[1]{\ThisStyle{%
		\vbox {\m@th\ialign{##\crcr
				\rightharpoonupfill \crcr
				\noalign{\kern-\p@\nointerlineskip}
				$\hfil\SavedStyle#1\hfil$\crcr}}}}

\def\rightharpoonupfill{%
	$\SavedStyle\m@th\mkern+0.8mu\cleaders\hbox{$\shortbar\mkern-4mu$}\hfill\rightharpoonuptip\mkern+0.8mu$}

\def\rightharpoonuptip{%
	\raisebox{\z@}[2pt][1pt]{\scalebox{0.55}{$\SavedStyle\rightharpoonup$}}}

\def\shortbar{%
	\smash{\scalebox{0.55}{$\SavedStyle\relbar$}}}
\makeatother

\makeatletter
\newcommand{\overlefharpoonup}[1]{\ThisStyle{%
		\vbox {\m@th\ialign{##\crcr
				\leftharpoonupfill \crcr
				\noalign{\kern-\p@\nointerlineskip}
				$\hfil\SavedStyle#1\hfil$\crcr}}}}

\def\leftharpoonupfill{%
	$\SavedStyle\m@th\mkern+0.8mu\cleaders\hbox{$\shortbar\mkern-4mu$}\hfill\leftharpoonuptip\mkern+0.8mu$}

\def\leftharpoonuptip{%
	\raisebox{\z@}[2pt][1pt]{\scalebox{0.55}{$\SavedStyle\leftharpoonup$}}}

\makeatletter
\newsavebox\myboxA
\newsavebox\myboxB
\newlength\mylenA

\newcommand*\xoverline[2][0.75]{%
	\sbox{\myboxA}{$\m@th#2$}%
	\setbox\myboxB\null
	\ht\myboxB=\ht\myboxA%
	\dp\myboxB=\dp\myboxA%
	\wd\myboxB=#1\wd\myboxA
	\sbox\myboxB{$\m@th\overline{\copy\myboxB}$}
	\setlength\mylenA{\the\wd\myboxA}
	\addtolength\mylenA{-\the\wd\myboxB}%
	\ifdim\wd\myboxB<\wd\myboxA%
	\rlap{\hskip 0.5\mylenA\usebox\myboxB}{\usebox\myboxA}%
	\else
	\hskip -0.5\mylenA\rlap{\usebox\myboxA}{\hskip 0.5\mylenA\usebox\myboxB}%
	\fi}
\makeatother

\DeclareSymbolFont{symbolsC}{U}{txsyc}{m}{n}
\SetSymbolFont{symbolsC}{bold}{U}{txsyc}{bx}{n}
\DeclareFontSubstitution{U}{txsyc}{m}{n}
\DeclareMathSymbol{\strictif}{\mathrel}{symbolsC}{74}
\DeclareMathSymbol{\strictfi}{\mathrel}{symbolsC}{75}

\usepackage{hyperref}
\usepackage{tikz}

\tikzset{vtx/.style={inner sep=1.7pt, outer sep=0pt, circle, fill=black,draw}}

\DeclareMathOperator{\conv}{conv}
\DeclareMathOperator{\cov}{cov}
\DeclareMathOperator{\Cov}{Cov}
\renewcommand{\tilde}{\widetilde}

\title{On Relative Ordered Tur\'an Density}

\author[D. King]{Dylan King}
\address{California Institute of Technology, Pasadena, USA}
\email{dking@caltech.edu}

\author[B. Lidick\'y]{Bernard Lidick\'y}
\address{Department of Mathematics, Iowa State University, Ames IA, USA}
\email{lidicky@iastate.edu}
\thanks{The second author was supported by NSF DMS-2152490 and Scott Hanna Professorship.}

\author[M. Ouyang]{Minghui Ouyang}
\address{School of Mathematical Sciences, Peking University, Beijing, China}
\email{ouyangminghui1998@gmail.com}

\author[F. Pfender]{Florian Pfender}
\address{University of Colorado, Denver, USA}
\email{florian.pfender@ucdenver.edu}
\thanks{The fourth author was supported by NSF DMS-2152498.}

\author[R. Wang]{Runze Wang}
\address{Department of Mathematical Sciences, University of Memphis, Memphis TN, USA}
\email{rwang6@memphis.edu}

\author[Z. Xiang]{Zimu Xiang}
\address{Department of Mathematics, University of Illinois Urbana-Champaign, Urbana IL, USA}
\email{zimux2@illinois.edu}
\thanks{The sixth author was supported in part by NSF RTG DMS-1937241.}

\begin{document}

\begin{abstract}
	For an ordered graph~$F$, denote the Tur\'an density by $\vec{\pi}(F)$. The relative Tur\'an density, denoted by $\rho(F)$, is the supremum over $\alpha \in [0,1]$ such that every ordered graph $G$ contains an $F$-free subgraph~$G'$ with~$e(G') \geq \alpha e(G)$. Reiher, R\"odl, Sales and Schacht~\cite{RRSS25} showed that $\rho(P) = \vec{\pi}(P)/2$ and $\rho(K) = \vec{\pi}(K)$ for any ascending path $P$ or clique $K$. They asked if there are any ordered graphs~$F$ with~$\vec{\pi}(F)/2 < \rho(F) < \vec{\pi}(F)$. We answer this question in the affirmative by describing a family of such~$F$. We also show that the relative Tur\'an densities of a large family of ordered matchings (including $\{\{1,6\}, \{2,3\}, \{4,5\}\}$ and $\{\{1,3\}, \{2,5\}, \{4,6\}\}$) are~$0$.
\end{abstract}


\maketitle

\section{Introduction}
For a graph $F$ and an integer $n$, the extremal number $\ex(n,F)$ is the maximum number of edges in an $F$-free $n$-vertex graph. Determining this number precisely is a challenge for most graphs $F$, and researchers have focused on the leading term by studying the Tur\'an density of $F$ defined as
\begin{align}
	\pi(F) = \lim_{n\to\infty} \frac{\ex(n,F)}{\binom{n}{2}}.\label{eq:dpi}
\end{align}
The Tur\'an densities of graphs are well understood through the formula
\begin{align}
    \pi(F) = 1 - \frac{1}{\chi(F)-1}\label{eq:pi},
\end{align}
due to Erd\H os, Stone and Simonovits~\cites{ES46, ES66}.

Pach and Tardos~\cite{PT06} developed an analogue of~\eqref{eq:pi} for ordered graphs. In this setting, each graph is equipped with a linear ordering of its vertex set, and every subgraph inherits the induced ordering. The extremal number for ordered graphs, denoted by~$\vec{\ex}(n, F)$, is defined as the maximum number of edges in an $n$-vertex ordered graph that contains no copy of $F$ (as an ordered subgraph). Just as in \eqref{eq:dpi}, one can define the \emph{ordered Tur\'an density}, 
	\[ \vec{\pi}(F) = \lim_{n\to\infty}\frac{\vec{\ex}(n,F)}{\binom{n}{2}}\,. \]
For an ordered graph $F$, the \emph{interval chromatic number}, denoted by $\chi_{<}(F)$, is the smallest $k$ such that $F$ has a proper $k$-vertex-coloring where each color class induces an interval in the vertex ordering. The aforementioned analogue of~\eqref{eq:pi}, established by Pach and Tardos~\cite{PT06}, states that
\begin{align}
    \vec{\pi}(F) = 1 - \frac{1}{\chi_{<}(F)-1} \label{eq:vpi}\,.
\end{align}
Observe that while the chromatic number of a graph is notoriously difficult to determine, the interval chromatic number is easily determined by a greedy search considering the vertices in order. Clearly, $\chi_<(F)\ge \chi(F)$ and therefore $\vec{\pi}(F)\ge \pi(F)$, and every graph has orderings where equality holds.

Let us return briefly to the unordered setting. A well known probabilistic argument shows that for any graph~$F$, every graph~$G$ contains a subgraph~$G'\subseteq G$ which is~$(\chi(F)-1)$-partite (and therefore~$F$-free) with~$e(G') \geq \pi(F) \cdot e(G)$. Furthermore, by considering the case~$G=K_n$ for large~$n$, this property fails when~$\pi(F)$ is replaced by any larger number. Reiher, R\"odl, Sales and Schacht~\cite{RRSS25} introduced the following definition.

\begin{definition}
	Given an ordered graph $F$, the \emph{relative Turán density} of $F$,~$\rho(F)$, is 
		\[ \sup\{ \alpha \in [0,1] \colon \text{every ordered $G$ has an $F$-free subgraph $G'$ with } e(G') \geq \alpha e(G)\}\,. \]
\end{definition}

Our preceding discussion shows that $\rho(F)=\pi(F)$ for unordered graphs, but in the ordered case we find more nuanced behavior. By considering again~$G=K_n$ for large~$n$ it follows that~$\rho(F) \leq \vec{\pi}(F) = \frac{\chi_<(F)-2}{\chi_<(F)-1}$, with equality whenever~$\chi(F)=\chi_<(F)$.

Let~$P_k$ be the monotone path on~$k$ vertices and~$\ell(F)$ denote the number of vertices\footnote{Here we depart from the notation used in~\cite{RRSS25} so that~$\ell(P_k) = \chi_<(P_k)=k$.} in the longest monotone path in~$F$. 

Reiher, R\"odl, Sales and Schacht~\cite{RRSS25} showed that

\begin{align}
	\rho(F) \geq \frac{\ell(F)-2}{2(\ell(F)-1)}\,.\label{eq:l}
\end{align}

For $P_k$, since $\ell(P_k) = \chi_{<}(P_k)$, the lower bound \eqref{eq:l} specifies to  $\rho(P_k) \geq \vec{\pi}(P_k)/2$. In this case (the primary result of their article) they proved equality $\rho(P_k) = \vec{\pi}(P_k)/2$. As noted above, any unordered graph $F$ has an ordering such that $\rho(F) = \pi(F)=\vec{\pi}(F)$. They~\cite{RRSS25} asked whether there are any ordered graphs $F$ satisfying 
\begin{align}
	\vec{\pi}(F)/2 < \rho(F) < \vec{\pi}(F)\,. \label{eq:question}
\end{align}
Our first result is a family of ordered graphs satisfying~\eqref{eq:question} which we introduce now. For  $a \geq 2$ and $b\geq1$, let $Q_{a,b}$ be the graph obtained from the monotone path on vertices $\{1,\dots,1+a+b\}$ by adding the edge $\{1,1+a\}$. To analyze~$\rho(Q_{a,b})$ it will be necessary to identify large ordered graphs~$G$ which are difficult (in the sense of edge deletion) to cleanse of~$Q_{a,b}$; we introduce these now. For $a \geq 2$ and $n-1$ a multiple of $a$, let $B_{a,n}$ be the union of monotone paths on vertices $\{1,2,\dots,n\}$ and $\{1,a+1,2a+1,\dots,n\}$. See Figure~\ref{fig:QB} for an example of $Q_{2,2}$ and $B_{2,9}$.
 
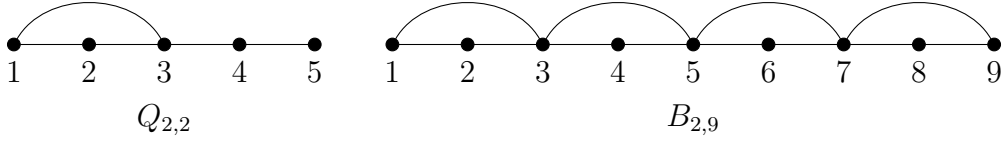
\begin{figure}[ht]
	\begin{center}
		\begin{tikzpicture}
    		\draw 
    		\foreach \i in {1,2,3,4,5}{
        		(\i,0)node[vtx,label=below:$\i$](v\i){}
    		}
    		(v1)--(v2)--(v3)--(v4)--(v5)
    		(v1) to[bend left=60](v3)
    		;
    		\draw (3,-1) node{$Q_{2,2}$};
		\end{tikzpicture}
		\hskip 1em
		\begin{tikzpicture}
    		\draw 
    		\foreach \i in {1,...,9}{
    			(\i,0)node[vtx,label=below:$\i$](v\i){}
    		}
    		(v1)--(v2)--(v3)--(v4)--(v5)--(v6)--(v7)--(v8)--(v9)
    		(v1) to[bend left=60](v3)
    		(v3) to[bend left=60](v5)
    		(v5) to[bend left=60](v7)
    		(v7) to[bend left=60](v9)
    		;
    		\draw (5,-1) node{$B_{2,9}$};
		\end{tikzpicture}
	\end{center}
	\caption{Graphs $Q_{2,2}$ and $B_{2,9}$.}
	\label{fig:QB}
\end{figure}

\begin{theorem}\label{thm:main}
    For integers $a \geq 2$ and $1 \leq b \leq a$, we have $\vec{\pi}(Q_{a,b})=\frac{a+b-1}{a+b}$ and $\rho(Q_{a,b}) \leq \frac{a}{a+1}$. In addition, if $a$ is even, then $1/2 \leq \rho(Q_{a,b})$. 
\end{theorem}
If $a$ is even and $b\geq 2$, Theorem~\ref{thm:main} implies that $Q_{a,b}$ satisfies \eqref{eq:question} as
	\[ \frac12 \cdot \vec{\pi}(Q_{a,b}) \,=\, \frac{a+b-1}{2(a+b)} \,<\, \frac{1}{2} \,\le\, \rho(Q_{a,b}) \,\le\, \frac{a}{a+1} \,<\, \frac{a+b-1}{a+b} \,=\, \vec{\pi}(Q_{a,b}). \]

Recall that for every ordered graph $F$,
	\[ \frac{\ell(F)-2}{2(\ell(F)-1)} \,\le\, \rho(F) \,\le\, \frac{\chi_<(F)-2}{\chi_<(F)-1} \,=\, \vec{\pi}(F). \]

Question \eqref{eq:question} in a sense replaces $\ell(F)$ by $\chi_<(F)$ in the lower bound. One may ask if each ordered graph $F$ satisfies $\vec{\pi}(F)/2 \leq \rho(F)$. The following proposition answers this question in the negative in a strong sense -- there are graphs with~$\vec{\pi}(F)$ arbitrarily close to~$1$ and relative Tur\'an density~$\rho(F)=0$. For $j \in \NN$, let $M_j$ be the ordered matching with vertices $[2j]$ and edges $\bigl\{\{2i-1,2i\} \colon i \in [j] \bigr\}$. 
\begin{prop}
    For every $j \geq 2$, $\chi_<(M_j)=j+1$ and $\rho(M_j)=0$.
\end{prop}
\begin{proof}
	It is immediate to see that $\chi_<(M_j)=j+1$ since any interval vertex-coloring of~$M_j$ has~$\chi_<(2i-1) \neq \chi(2i)$ for each~$i \in [j]$. To show $\rho(M_j)=0$, consider the graphs $M_k$ with $k > j$. Since any $j$ edges in $M_k$ induce a copy of $M_j$, any~$G' \subseteq M_k$ which is~$M_j$-free has~$e(G') < j = \frac{j}{k} \cdot e(M_k)$. Letting $k \to \infty$ implies $\rho(M_j)=0$.
\end{proof}

The previous proposition can be extended to a more general observation. Suppose that~$F$ is an ordered graph with vertex set~$[n]$ and~$I \in \binom{[2n]}{n}$. Then denote by~$F+_I F$ the ordered graph on vertex set~$[2n]$ consisting of two vertex-disjoint copies of~$F$; one on~$I$ and one on~$[2n] \setminus I$, each maintaining their original ordering.

\begin{obs}\label{obs:copies}
    Let $F$ be an ordered graph on~$[n]$ and~$I \in \binom{[2n]}{n}$. Suppose either that
    \begin{enumerate}[itemsep=0pt, parsep=0pt, topsep=0pt, partopsep=0pt, left=10pt]
        \item  $|I \cap \{2i-1, 2i\}|=1$ for every~$i \in [n]$, or
        \item $I = [n]$.
    \end{enumerate}
    Then $\rho(F+_I F) = \rho(F)$.
\end{obs}

In the first case, $F+_I F$ is a subgraph of the ordered blow-up $F^{(2)}$ of $F$, and the result follows from~\cite{RRSS25}, which shows that~$\rho(F)$ is invariant under blowups. In the second case, for any $\eps > 0$, let $G$ be a graph such that every subgraph of $G$ with more than $(\rho(F) + \eps)e(G)$ edges contains a copy of $F$. We place sufficiently many copies of $G$ in sequence, disjoint from one another. Then, in any $F$-free subgraph of the resulting graph, there can be at most one copy of $G$ whose subgraph has edge density exceeding $\rho(F) + \eps$. The above observation can be generalized to $k$ copies of $F$.

The class of ordered~$F$ with~$\vec{\pi}(F)=0$ is exactly those with~$\chi_<(F)=2$, and is a (strict, by~$M_j$ given above) subset of those~$F$ for which~$\rho(F)=0$. It is natural to ask for an analogous characterization of those~$F$ with~$\rho(F)=0$, and a natural first step is to decide if~$\rho(M)=0$ for every ordered matching~$M$. We confirm this for the following three-edge matchings.

\begin{theorem}\label{thm:16-23-45}
	Let $M = \bigl\{ \{1,6\}, \{2,3\}, \{4,5\} \bigr\}$. Then $\rho(M) = 0$. 
\end{theorem}

\begin{theorem} \label{thm:13-25-46}
    Let $M = \bigl\{ \{1,3\}, \{2,5\}, \{4,6\} \bigr\}$. Then $\rho(M) = 0$. 
\end{theorem}

Our final result requires some new notation. If~$F$ is an ordered graph on vertices~$\{1,\dots,n\}$, the local extension ~$\tilde{F}$ is the ordered graph on vertex set~$\{1,\dots,3n\}$ with edges between~$3i-1$ and~$3j-1$ whenever~$ij\in E(F)$ and between~$3k-2$ and~$3k$ for each~$k \in [n]$. Informally, we take a copy of~$F$ and add a short edge `over' each vertex. See Figure~\ref{fig:loc-ext} for an illustration of $\tilde{Q_{2,2}}$. 

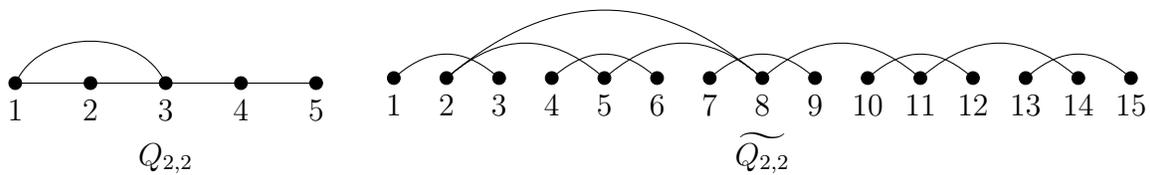
\begin{figure}[ht]
	\begin{center}
		\begin{tikzpicture}
			\draw 
			\foreach \i in {1,2,3,4,5}{
				(\i,0)node[vtx,label=below:$\i$](v\i){}
			}
			(v1)--(v2)--(v3)--(v4)--(v5)
			(v1) to[bend left=60](v3);
			\draw (3,-1) node{$Q_{2,2}$};
		\end{tikzpicture}
		\hskip 1em
		\begin{tikzpicture}[xscale = 0.7]
			\draw 
			\foreach \i in {1,...,15}{
				(\i,0)node[vtx,label=below:$\i$](v\i){}
			};
			\draw
			\foreach  \u/\v in {2/5, 5/8, 8/11, 11/14, 2/8}{
				(v\u) to[bend left] (v\v)
			};
			\draw
			\foreach  \u/\v in {1/3, 4/6, 7/9, 10/12, 13/15}{
				(v\u) to[bend left] (v\v)
			};
			\draw (8,-1) node{$\tilde{Q_{2,2}}$};
		\end{tikzpicture}
	\end{center}
	\caption{Ordered graphs~$Q_{2,2}$ and $\tilde{Q_{2,2}}$.}
	\label{fig:loc-ext}
\end{figure}

We prove that~$\rho$ is invariant under local extension.
\begin{theorem} \label{thm:loc-ext} 
	For every ordered graph $F$, we have $\rho(\tilde{F}) = \rho(F)$. 
\end{theorem}

Starting from a single edge~$\{1,2\}$ and iterating this procedure we see that $\rho(M) = 0$ for a family of matchings, which in particular includes the two examples of Theorems~\ref{thm:16-23-45} and~\ref{thm:13-25-46}.

In the next section, we present proofs of Theorems~\ref{thm:main} and \ref{thm:16-23-45}. In Section~\ref{sec:13-25-46} we prove Theorem~\ref{thm:13-25-46} and in Section~\ref{sec:loc-ext} we prove Theorem~\ref{thm:loc-ext}. Although Theorem~\ref{thm:loc-ext} implies Theorems~\ref{thm:16-23-45} and~\ref{thm:13-25-46}, we include short direct proofs because they utilise different methods. We conclude the paper with some unresolved questions.

\section{Proofs of Theorems~\ref{thm:main} and \ref{thm:16-23-45}}

\begin{proof}[Proof of Theorem~\ref{thm:main}]
	Since $P_{a+b+1}\subset Q_{a,b}$, we have $\chi_{<}(Q_{a,b}) = a+b+1$. Applying  \eqref{eq:vpi} yields $\vec{\pi}(Q_{a,b})=\frac{a+b-1}{a+b}$. Suppose further that $a$ is even, and let $G$ be an arbitrary ordered graph. Since $Q_{a+b}$ contains an odd cycle, any largest bipartite subgraph~$G'$ of~$G$, with~$e(G') \geq \frac{1}{2} e(G)$, avoids $Q_{a+b}$. Hence $1/2 \leq \rho(Q_{a+b})$, and it remains only to show that~$\rho(Q_{a,b}) \leq \frac{a}{a+1}$.

	Consider $B_{a,n}$ with~$n = \ell a+1$ for some sufficiently large integer $\ell$. Let $E\subseteq E(B_{a,n})$ be any set of edges such that $B_{a,n} - E$ is $Q_{a,b}$-free. The idea behind the proof is that on average, every cycle in~$B_{a,n}$ must contain close to one edge of~$E$; otherwise we would be able to extend an intact cycle by a monotone path and build a copy of~$Q_{a,b}$ in~$B_{a,n} - E$. Let $C^i$ be the cycle in $B_{a,n}$ starting at vertex $(i-1)a+1$ and ending at vertex~$ia+1$. Denote the vertices of $C^i$ by $\{C^i_1,\ldots,C^i_{a+1}\}$, let~$x_i = |E \cap E(C^i)|$, and let~$i_1 < i_2 < \dots < i_k$ enumerate those~$i \in [\ell]$ with~$x_i=0$ (allowing potentially for~$k=0$ and the list to be empty).

	We claim that 
		\begin{equation}\label{eqn:missing_edges} 
			\sum_{i=i_{m}}^{i_{m+1}-1}x_i \geq (i_{m+1}-i_{m})
		\end{equation} 
	for all~$m \in [k]$. Suppose, for the sake of contradiction, that~\eqref{eqn:missing_edges} fails for some~$m \in [k]$. We will show this forces the existence of a copy of~$Q$. Since the~$x_i$ are nonnegative integers it follows that~$x_i = 1$ for each~$i_{m} < i < i_{m+1}$; that is, from each such~$C^i$ there is exactly one edge of $E$ in $C^{i}$. Since each~$C^{i}$ is a cycle, there exists an ascending path from $C^{i}_{1}$ to $C^{i}_{a+1}$ avoiding $E$ (namely either the single edge~$\{C^{i}_{1},C^{i}_{a+1} \}$ or the path~$\{C^{i}_{1},C^{i}_{2},\dots,C^{i}_{a+1}\}$). Hence, by concatenating these paths, there is an ascending path $P$ that avoids the edges of $E$ starting at $C^{i_{m}}_{a+1}$ and ending at $C^{i_{m+1}}_1$. Since $C^{i_{m+1}}$ has no edges in $E$, we can extend $P$ by $C^{i_{m+1}}_1, C^{i_{m+1}}_2,\ldots,C^{i_{m+1}}_{a+1}$ while still avoiding $E$. As $a \geq b$, there is a copy of $Q_{a,b}$ in $C^{i_m} \cup P$, and therefore by contradiction~$\eqref{eqn:missing_edges}$ holds for each~$m \in [k]$. Therefore,
	\begin{align*}
		|E| = \sum_{i=1}^\ell x_i&=\sum_{j<i_1}x_j+ \sum_{j \ge i_k}x_j +\sum_{m=1}^{k-1}\sum_{j=i_m}^{i_{m+1}-1}x_j\\
			&\geq (i_1-1)+(\ell-i_k)+\sum_{m=1}^{k-1}(i_{m+1}-i_m)\\
			&= \ell -1\,.
	\end{align*}

	Since $|E(B_{a,n})| = (a+1)\ell$,
		\[ \rho(Q_{a,b}) \leq \frac{(a+1)\ell - (\ell-1)}{(a+1)\ell} = \frac{a}{a+1} + \frac{1}{(a+1)\ell}. \]
	The result follows by taking $\ell$ arbitrarily large.
\end{proof}

\begin{proof}[Proof of Theorem~\ref{thm:16-23-45}]
	For $d \in \NN$, we define the graph $H_d$ recursively as follows:
	\begin{itemize}[itemsep=0pt, parsep=0pt, topsep=0pt, partopsep=0pt, left=0pt]
    	\item For $d = 1$, let $H_1 \coloneq \bigl\{ \{1, 2\} \bigr\}$ be the single ordered edge.
    	\item For $d \ge 2$, let $H_{d-1}^1$ and $H_{d-1}^2$ be two disjoint copies of $H_{d-1}$, arranged so that all vertices of $H_{d-1}^1$ appear before all vertices of $H_{d-1}^2$ in the ordering. Introduce $2^{d-1}$ new vertices $\{a_1, \cdots, a_{2^{d-1}}\}$ placed before $H_{d-1}^1$, and another $2^{d-1}$ vertices $\{b_1, \cdots, b_{2^{d-1}}\}$ placed after $H_{d-1}^2$. Let 
        	\[ M_d \coloneq \bigl\{ \{a_1, b_{2^{d-1}}\}, \{a_2, b_{2^{d-1}-1}\}, \cdots, \{a_{2^{d-1}}, b_1\} \bigr\}, \]
    	and set $H_d \coloneq H_{d-1}^1 \cup H_{d-1}^2 \cup M_d$. 
	\end{itemize}

	Graphs $M = \bigl\{ \{1,6\}, \{2,3\}, \{4,5\} \bigr\}$ and $H_3$ are shown in Figure~\ref{fig:M}.

	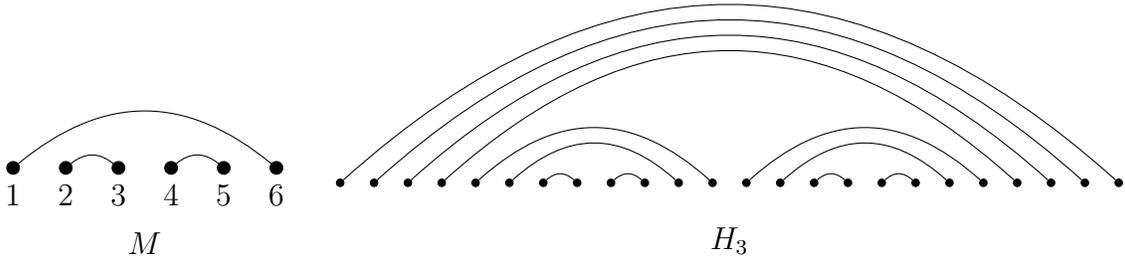
\begin{figure}[ht]
		\begin{center}
			\begin{tikzpicture}[xscale = 0.7]
    			\draw
    			\foreach \i in {1,...,6}{
    				(\i,0)node[vtx,label=below:$\i$](v\i){}
    			}
    			;

				\draw
				\foreach  \u/\v in {1/6,2/3,4/5}{
					(v\u) to[bend left] (v\v)
				}
				;
				\draw(3.5,-1)node{$M$};
			\end{tikzpicture}
			\hskip 1em
			\begin{tikzpicture}[xscale = 0.45, yscale = 0.7]
    			\draw
    			\foreach \i in {1,...,24}{
    				(\i,0)node[circle, draw, fill=black, inner sep=1pt](v\i){}
    			}
    			;
    
				\draw
				\foreach  \u/\v in {1/24,2/23,3/22,4/21,5/12,6/11,7/8,9/10,13/20,14/19,15/16,17/18}{
					(v\u) to[bend left] (v\v)
				}
				;
				\draw (12.5,-1.1)node{$H_3$};
			\end{tikzpicture}
		\end{center}
		\caption{Graphs $M$ and $H_3$.}\label{fig:M}
	\end{figure}

	It is easy to see that $|E(H_d)| = d \cdot 2^{d-1}$. We show by induction on~$d$ that any $M$-free subgraph~$F\subseteq H_d$ has~$e(F) \leq \frac{2}{d}e(H_d)$. The base cases $d = 1$ and $d = 2$ are clear. 

	Suppose now that the claim holds for $d-1$, and let $F \subseteq H_d$ be an $M$-free subgraph. If no edge of $M_d$ is included in $F$, then by the inductive hypothesis, the number of edges in $F$ is at most 
		\[ \frac{2}{d-1} \cdot |E(H_{d-1}^1)| + \frac{2}{d-1} \cdot |E(H_{d-1}^2)| = 2^d = \frac{2}{d} \cdot |E(H_d)|. \] 
  
	If $F$ contains at least one edge from $M_d$, then by the structure of the forbidden configuration, we cannot simultaneously select edges from both $H_{d-1}^1$ and $H_{d-1}^2$ without forming a copy of $M$. Applying the inductive hypothesis to that part (WLOG,~$H_{d-1}^1$) shows that the number of edges in~$F$ is at most 
		\[ |M_d|+\frac{2}{d-1}\cdot |E(H_1)|  = 2^d = \frac{2}{d} \cdot |E(H_d)|. \]
	Letting~$d \to \infty$ completes the proof that~$\rho(M)=0$.
\end{proof}

\section{Proof of Theorem~\ref{thm:13-25-46}} \label{sec:13-25-46}
We need the following construction of a quasi-random ordered matching, which is a slight modification of a construction due to Arman, R\"{o}dl, and Sales~\cite{ARS22}. For certain notational purposes it will be more convenient to construct a graph on the vertex set~$[0,2]\subset\RR$. Unless otherwise specified, an interval may be open, closed, or half-open. 

\begin{lemma} \label{lem:quasi-random}
    For any $\eps > 0$, there exists an integer $N = N(\eps)$ such that for all $n > N$, there exists an interval-bipartite matching, denoted $G(n, \eps)$, with $n$ edges between the two parts $(0,1)$ and $(1,2)$ satisfying
        \[ \forall\ \text{intervals } I \subseteq (0,1),\ J \subseteq (1,2), \quad \Bigl|e(I,J) - |I| \cdot |J| \cdot n \Bigr| \le \eps n, \]
    where $|I|$ denotes the length of $I$, and $e(I, J)$ denotes the number of edges between points in $I$ and $J$.
\end{lemma}

\begin{proof}
    Let $t = \lceil 50 / \eps \rceil$, and partition the intervals $(0,1)$ and $(1,2)$ into $t$ equal-length subintervals. For any interval $I \subseteq (0,1)$ (and similarly for $J \subseteq (1,2)$), let $I_1$ and $I_2$ denote the largest and smallest unions of these discretized subintervals such that $I_1 \subseteq I \subseteq I_2$. Then, it is straightforward to verify that
        \[ 0 \le e(I, J) - e(I_1, J_1),\ e(I_2, J_2) - e(I, J) \le e(I_2 \setminus I_1, J_1) + e(I_1, J_2 \setminus J_1) + e(I_2 \setminus I_1, J_2 \setminus J_1)\,, \]
    and therefore (by applying routine set-theoretic calculus) it suffices to prove that
        \[ \Bigl| e(I, J) - |I| \cdot |J| \cdot n \Bigr| \le \frac{\eps n}{10} \]
    for all $I$ and $J$ that are unions of consecutive discretized intervals.

    We construct the edge set $E(G) \subseteq (0,1) \times (1,2)$ by selecting $n$ edges independently, where each edge connects a pair of endpoints chosen uniformly at random from $(0,1)$ and $(1,2)$, respectively. By standard concentration inequalities (e.g., Hoeffding’s inequality), for any fixed pair of intervals $I \subseteq (0,1)$ and $J \subseteq (1,2)$, we have
        \[ \mathbb{P} \biggl( \Bigl| e(I, J) - |I| \cdot |J| \cdot n \Bigr| \ge \frac{\eps n}{10} \biggr) \le \exp\left( - \frac{\eps^2 n}{50} \right). \]

    Since there are at most $t^4$ such interval pairs $(I, J)$ formed by unions of consecutive subintervals, a union bound shows that with high probability, the bound holds simultaneously for all such pairs. Hence, for sufficiently large $n$, there exists a graph $G$ satisfying the desired property. 
\end{proof}
We will need the following, which extends the quasirandom property obtained above to finite unions of intervals.
\begin{prop} \label{prop:quasi-uniform}
    Suppose that $I \subseteq (0,1)$ and $J \subseteq (1,2)$ are disjoint unions of $a$ and $b$ intervals, respectively. Then in $G(n, \eps)$, we have
        \[ \Bigl| e(I, J) - |I| \cdot |J| \cdot n \Bigr| \le ab \eps n. \]
\end{prop}

We construct the graph witnessing $\rho(M) = 0$ in Theorem~\ref{thm:13-25-46} as follows. Take~$\epsilon>0$ small and $n>N(\epsilon)$ sufficiently large, and let $G_1 = G(n, \eps)$. We retain only the points in $(0,2)$ that appear as endpoints of edges in $G_1$, preserving their original order in $(0,2)$. 

For each $d \ge 1$, define
    \[ G_{d+1} = \tfrac{1}{2} \cdot G_d \cup \left(1 + \tfrac{1}{2} \cdot G_d \right) \cup G(2^d n, \eps), \]
where $\tfrac{1}{2} \cdot G_d$ denotes the rescaling of $G_d$ by a factor of $\tfrac{1}{2}$, and $1 + \tfrac{1}{2} \cdot G_d$ denotes its rightward translation by $1$. It follows that $e(G_d) = d \cdot 2^{d-1} n$.

Now we give the key definition and heuristic for this construction. Given a subgraph $H \subseteq G_d$, we say that a point $x \in (0,2)$ is \emph{covered} if there is an edge~$\{a,b\} \in E(H)$ with $a < x < b$. Intuitively, most points will be covered, and most edges between two covered points must be removed. The remainder of the proof is dedicated to formalizing this notion. 

\begin{lemma} \label{lem:intvl-count}
    For any subgraph $H \subseteq G_d$, the set of covered points is a union of at most $2^d - 1$ disjoint open intervals.
\end{lemma}

\begin{proof}
    We proceed by induction on $d$. By construction,
        \[ G_{d+1} = \tfrac{1}{2} \cdot G_d \cup \left( 1 + \tfrac{1}{2} \cdot G_d \right) \cup G(2^d n, \eps). \]
    By the inductive hypothesis, the sets of points covered by any subgraph of either of the first two components can be written as $2^d - 1$ disjoint open intervals. Finally, if~$y$ denotes the leftmost left endpoint of an edge in~$H \cap G(2^dn,\epsilon)$, and~$z$ denotes the rightmost right endpoint of any edge, then the third component covers a point~$x$ if and only if~$y<x<z$, contributing at most one additional interval.
    Therefore, the set of covered points may be written as the union of at most
        \[ (2^d - 1) + (2^d - 1) + 1 = 2^{d+1} - 1, \]
     open intervals, completing the induction.
\end{proof}

\begin{proof}[Proof of Theorem~\ref{thm:13-25-46}]
	For~$d \in \NN$ and~$t \in [0,1]$, define
    	\[ f_d(t) = 
        	\begin{cases}
            	2^{d-1} t^2 & \text{if } t \in \left[0, \frac{1}{2^{d-1}}\right], \\
            	2t - \frac{1}{2^{d-1}} & \text{if } t \in \left(\frac{1}{2^{d-1}}, 1\right],
        	\end{cases} \]
	and observe that, for fixed~$d$,~$f_d(t)$ is convex and increasing in~$t$. It may be helpful for intuition to also note that the~$f_d$ converge in~$d$ to~$f(t) = 2t$. For a subgraph $H \subseteq G_d$, let $C(H)$ denote the set of points covered by $H$, and let $|C(H)|$ denote its total length, i.e., the sum of the lengths of all intervals that make up $C(H)$. Furthermore let~$L(H) =  \{x \colon \{x,y\} \in E(H)\}$ denote the set of left endpoints edges in~$H$,~$R(H) = \{y \colon \{x,y\} \in E(H)\}$ denote the set of right endpoints, and for a set~$X \subseteq (0,2)$, let~$\conv(X)$ denote the convex hull. Finally, for~$A \subset (0,2)$, let~$H_A = \{\{x,y \} \colon x,y \in A \}$ be the induced subgraph taken on~$A$. We aim to prove the following: 

	\textbf{Claim.} Let $d \in \NN$, $t \in \RR_{\ge 0}$, and suppose $H \subseteq G_d$ is an $M$-free subgraph with $|C(H)| \le 2t$. Then
    	\[ e(H) \le f_d(t) \cdot 2^{d-1} n + { 10^d} \eps n. \]

	\begin{proof}[Proof of the Claim]
		We proceed by induction on $d$. 
    
    	\textit{Base case ($d = 1$):} Let $H \subseteq G_1 = G(n, \eps)$, set~$I = S(H_{[0,1)})$, and set $J = S(H_{(1, 2]})$. By the quasi-randomness of $G(n,\eps)$,
        	\[ e(H) \le |E(I, J)| \le \left(|I| \cdot |J| + \eps\right)n. \]
    	Since~$(I \cup J) \subseteq C(H)$ and~$I \cap J = \emptyset$, we have $|I| + |J| \le |C(H)| \le 2t$ so that $|I| \cdot |J| \leq t^2$. Thus,
    		\[ e(H) \le \left(t^2 + \eps\right) n \le (f_1(t) + \eps) n. \]

    	\textit{Induction step:} Suppose the claim holds for some $d \ge 1$, and consider an $M$-free subgraph $H \subseteq G_{d+1}$. Decompose the edges of $H$ into three parts:
        \begin{itemize}[itemsep=0pt, parsep=0pt, topsep=0pt, partopsep=0pt]
            \item $A=H_{[0,1)}$: edges from the left copy $\tfrac{1}{2} \cdot G_d$,
            \item $B=H_{(1,2]}$: edges from the right copy $1 + \tfrac{1}{2} \cdot G_d$,
            \item $C$: the remaining edges, i.e. those from $G(2^d n, \eps)$. 
        \end{itemize}

    	Let $I = C(A) \subseteq (0, 1)$ and $J = C(B) \subseteq (1, 2)$ denote the sets of covered points by $A$ and $B$, respectively. Then $I, J$ are each unions of at most $2^d - 1$ intervals from Lemma~\ref{lem:intvl-count}, and their complements may each be expressed as the union of at most $2^d$ intervals, say~$(0,1) \setminus I = \bigdcup_i P_i$ and~$(1,2) \setminus J = \bigdcup_j Q_j$. Let~$I' \coloneq \bigdcup_i \conv(P_i \cap L(C))$ and~$I' \coloneq \bigdcup_j \conv(Q_j \cap R(C))$, so that each is the union of at most $2^d$ intervals. 

		Since $H$ is $M$-free, there are no edges in $C$ connecting $I$ and $J$. Thus,
        	\[ C \subseteq (I' \times J') \cup (I' \times J) \cup (I \times J'). \]

    	Clearly, $I \cup J \cup \mathring{I'} \cup \mathring{J'} \subseteq C(H)$, where $\mathring{I'}$ denotes the interior of $I'$. Let $a = |I|$, $b = |J|$, $x = |I'|$, and $y = |J'|$. Then
        	\[ a+b+x+y \le |C(H)| \le 2t, \quad a+x \le 1, \quad b+y \le 1. \]
    	By the inductive hypothesis and Proposition~\ref{prop:quasi-uniform}, we obtain
    	\begin{align*}
        	e(H) &= |A| + |B| + |C| \\
            	&\le \left( f_d(a) \cdot 2^{d-1} + { 10^d}\eps \right) n + \left( f_d(b) \cdot 2^{d-1} + { 10^d}\eps \right) n + (xy+bx+ay + 3 \cdot 2^d \cdot 2^d \eps) e(G(2^dn, \eps)) \\
            	&\le \left( \frac{f_d(a) + f_d(b)}{2} + xy + bx + ay \right) \cdot 2^d n + (10^d + 10^d + 3 \cdot 2^{3d}) \eps n \\
            	&\le \left( \frac{f_d(a) + f_d(b)}{2} + (a + x)(b + y) - ab \right) \cdot 2^d n + { 10^{d+1}} \eps n. 
    	\end{align*}

    	Denote $\frac{f_d(a) + f_d(b)}{2} + (a + x)(b + y) - ab$ by $(\ast)$. It suffices to show that $(\ast)\le f_{d+1}(t)$.

    	We now consider two cases based on the value of $t$:

    	\begin{itemize}[itemsep=0pt, parsep=0pt, topsep=0pt, partopsep=0pt, left=0pt]
       	 	\item $t \le \frac{1}{2^d}$. 
            	In this case, $a, b \le \frac{1}{2^{d-1}}$, so by definition,
                	\[ f_d(a) \,\le\, 2^{d-1} a^2, \quad f_d(b) \,\le\, 2^{d-1} b^2. \]
            	Hence,
                	\[ (\ast) \,\le\, \tfrac{1}{2} \cdot 2^{d-1} (a^2 + b^2) + (a + x)(b + y) - ab \,\le\, 2^{d-2} (a + b + x + y)^2 \,\le\, 2^d t^2 = f_{d+1}(t). \]
                
        	\item $\frac{1}{2^d} < t \le 1$.  
            	Since $f_d$ is convex, we may apply the following rebalancing:
                	\[ (a, b, x, y) \to (a + \Delta, b - \Delta, x - \Delta, y + \Delta), \]
            	without decreasing the value of $(\ast)$. Eventually, either $b = 0$ or $x = 0$.
            	\begin{itemize}[itemsep=0pt, parsep=0pt, topsep=0pt, partopsep=0pt, left=0pt]
					\item If $b = 0$, then by the monotonicity of $f_d$, we may assume without loss of generality that $x = 0$, after replacing $(x,a)$ by $(0, a+x)$. Then
							\[ (\ast) \,\le\, \frac{1}{2} \cdot f_d(a) + ay. \]
                    	If $a \ge \frac{1}{2^{d-1}}$, then using the linear part of $f_d$:
                        	\[ (\ast) \,\le\, \frac{1}{2} \left( 2a - \frac{1}{2^{d-1}} \right) + a(2t - a) \,=\, 2t - (2t-a)(1-a) - \frac{1}{2^d} \,\le\, 2t - \frac{1}{2^d} \,=\, f_{d+1}(t), \]
                    	where the last inequality follows from $a \le \min\{2t, 1\}$. 

                    	If $a < \frac{1}{2^{d-1}}$, then
                        	\[ (\ast) \,\le\, 2^{d-2} a^2 + a(2t - a) \,=\, (2^{d-2} - 1)a^2 + 2ta, \]
                    	which is monotone increasing in $a$, so we may apply the previous argument for when~$a=\frac{1}{2^{d-1}}$. 
                    
                \item If $x = 0$, then
                        \[ (\ast) \,\le\, \frac{1}{2} \left( f_d(a) + f_d(b) \right) + a(b + y) - ab \,=\, \frac{1}{2} \left( f_d(a) + f_d(b) \right) + a y. \]
                    Again, by the convexity of $f_d$, one can reduce to the previous case where $b = 0$. 
            \end{itemize}
    \end{itemize}

    	This completes the induction. 
	\end{proof}

	By the claim, for any $M$-free subgraph $H \subseteq G_d$, we have
    	\[ e(H) \,\le\, f_d(1) \cdot 2^{d-1} n + 10^d \eps n \,<\, \left( \frac{2}{d} + \frac{10^d}{d2^{d-1}}\eps \right) \cdot e(G_d).\]
	Letting $d \to \infty$ and $\eps \to 0$ (in that order, and always taking~$n>N(\epsilon)$ large enough), we conclude that $\rho(M) = 0$.
\end{proof}

\section{A Local Extension Argument} \label{sec:loc-ext}
In this section we prove Theorem~\ref{thm:loc-ext}. The following proposition plays the same role as Lemma~\ref{lem:quasi-random} in the proof of Theorem~\ref{thm:13-25-46}. 
\begin{lemma} \label{lem:quasi-random-2}
	Let $\Gamma$ be an ordered graph and let $\eps > 0$. For every sufficiently large integer $n$ divisible by $e(\Gamma)$, there exists a graph $G(n, \Gamma, \eps)$ on the vertex set $(0, v(\Gamma))$ with $n$ edges such that:
	\begin{itemize}[itemsep=0pt, parsep=0pt, topsep=0pt, partopsep=0pt, left=10pt]
		\item $G(n, \Gamma, \eps)$ is the disjoint union of $\tfrac{n}{e(\Gamma)}$ copies of $\Gamma$, and  
		\item for every interval $I \subseteq (0, v(\Gamma))$, the number of edges incident to $I$ is at most $(|I| + \eps)n$.
	\end{itemize}
\end{lemma}

\begin{proof}
	Partition the interval $(0, v(\Gamma))$ into subintervals
		\[ V_1 = (0,1], \quad V_2 = (1,2], \quad \cdots, \quad V_{v(\Gamma)} = (v(\Gamma)-1, v(\Gamma)). \]
	Let $V(\Gamma) = \{a_1, \dots, a_{v(\Gamma)}\}$ be the vertex set of $\Gamma$ ordered from left to right. Independently sample $\tfrac{n}{e(\Gamma)}$ copies of $\Gamma$ by placing each $a_i$ uniformly at random in the interval $V_i$ for every copy. The rest of the proof follows from the same argument as in Lemma~\ref{lem:quasi-random}. 
\end{proof}
Now we are ready to prove Theorem~\ref{thm:loc-ext}. 

\begin{proof}
	Clearly $\rho(\tilde{F}) \ge \rho(F)$ since~$F$ is an ordered subgraph of~$\tilde{F}$. To show the reverse inequality, fix $\delta > 0$ and suppose $\Gamma$ is a witness to the value of~$\rho(F)$ so that any~$\Gamma' \subseteq \Gamma$ which is $F$-free has $e(\Gamma') \leq \left(\rho(F) + \delta\right)e(\Gamma)$. Our goal is to construct a graph witnessing that $\rho(\tilde{F}) < \rho(F) + 2\delta$. 
	
	Let $m = v(\Gamma)$. Define 
		\[ G_1 \coloneq G(n, \Gamma, \eps), \]
	and recursively set
		\[ G_{d+1} \coloneq \biggl( \bigcup_{i = 0}^{m-1} \left(i + \tfrac{1}{m} \cdot G_d \right) \biggr) \cup G(m^dn, \Gamma, \eps). \]
    Then $G_d$ has $d m^{d-1} n$ edges and can be viewed as the union of $d$ layers, where in the $i$-th layer, there are $m^{d-i}$ translations of $\tfrac{1}{m^{d-i}} \cdot G(m^{i-1} \cdot n, \Gamma, \eps)$. We denote them by $G_d^{i,1}, \cdots, G_d^{i,m^{d-i}}$ and write $G_d^i = \bigcup_{j = 1}^{m^{d-i}} G_d^{i,j}$ for the $i$-th layer. 
	
	Let $H \subseteq G_d$ be an $\tilde{F}$-free subgraph, and let $H_{\le i} \subseteq H$ denote the subgraph consisting of all edges contained in the first $i$ layers (counted from the bottom). For an integer $i \ge 0$, define $\Cov(i) \subseteq (0,m)$ to be the set of points covered by edges of $H_{\le i}$, that is, $\Cov(i)$ consists of all points $x \in (0,m)$ for which there exists an edge $e \in E(H_{\le i})$ such that $x$ lies between the two endpoints of $e$. By the construction of $G_d$, the set $\Cov(i)$ is the union of at most $i \cdot m^d$ intervals. Let $\cov(i)$ denote the total length of $\Cov(i)$. For convenience, we set $\Cov(i) = \varnothing$ whenever $i \le 0$. 
	
	Clearly,  
		\[ 0 = \cov(0) \le \cov(1) \le \cdots \le \cov(d) \le m. \]
        
	Decompose the edge set of $H$ as $E(H) = A \cup B$, where
	\begin{itemize}[itemsep=0pt, parsep=0pt, topsep=0pt, partopsep=0pt]
		\item $A$ consists of edges whose endpoints both lie in $\Cov(i-1)$ for some $i$, and
		\item $B = E(H) \setminus A$ consists of the remaining edges. 
	\end{itemize}

    For each $i$, let $B_i \coloneq B \cap E(G_d^i)$ be the set of edges in $B$ that lie in the $i$-th layer. Then every edge in $B_i$ has at least one endpoint outside $\Cov(i-1)$. Moreover, the closure of $\Cov(i)$ contains the endpoints of all edges in $B_i$. Hence, by the second property of Lemma~\ref{lem:quasi-random-2}, we obtain
    \begin{align*}
        |B| \,\le\, \sum_{i=1}^d |B_i| &\le\, \sum_{i=1}^d \bigl(\cov(i) - \cov(i-1) + 3i \cdot m^d \eps \bigr) m^{d-1} n \\
            &<\, \cov(d) \cdot m^{d-1} n + 3d m^{2d} \eps n \\
            &\le\, m^d n + 3d m^{2d} \eps n, 
    \end{align*}
    where the first inequality comes from the fact that $\Cov(i) \setminus \Cov(i-1)$ is the union of at most $3i \cdot m^d$ intervals. 
    
	Note that any copy of $F$ contained in $A$ can always be extended to a copy of $\tilde{F}$ by attaching appropriate edges from the lower layers around each of its vertices (since they are covered by edges from lower layers), thus forming a copy of $\tilde{F}$ inside $H$. Therefore, $A$ is $F$-free. 
        
	Partition $A$ into blocks: 
        \[ A = A_d^1 \cup (A_{d-1}^1 \cup \cdots \cup A_{d-1}^m) \cup \cdots \cup (A_1^1 \cup \cdots A_1^{m^{d-1}}) \]
    where $A_i^j = A \cap G_d^{i,j}$. 
    
	Since each $G_d^{i,j}$ is a disjoint union of copies of $\Gamma$, any $F$-free subgraph of $G_d^{i,j}$ has relative density at most $\rho(F) + \delta$. Hence, for any $i,j$, $|A_i^j| \le (\rho(F) + \delta) \cdot e(G_d^{i,j}) = (\rho(F) + \delta) \cdot m^{i-1}n$. Summing over all blocks, we obtain
		\[ |A| \,=\, \sum_{i = 1}^d \sum_{j = 1}^{m^{d-i}} |A_i^j| \,<\, \sum_{i=1}^{d} m^{d-i} \cdot (\rho(F) + \delta) m^{i-1} n \,=\, (\rho(F) + \delta) d m^{d-1} n. \]
	
	Combining the bounds for $|A|$ and $|B|$, we obtain
		\[ |A| + |B| \,<\, (\rho(F) + \delta) d m^{d-1} n + m^d n + 3d m^{2d} \eps n. \]
	By choosing $d$ sufficiently large and $\eps$ sufficiently small so that 
        \[ m^d n + 3d m^{2d} \eps n < \delta d m^{d-1} n, \]
    we have
		\[ |A| + |B| \,<\, (\rho(F) + 2\delta) d m^{d-1} n = (\rho(F) + 2\delta)e(G_d). \]
	This completes the proof. 
\end{proof}

We give the following Corollary as a simple application of Theorem~\ref{thm:loc-ext}

\begin{cor}
    For every integer $k \ge 1$, the matching $M_k = \bigl\{ \{1,3\}, \{2,5\}, \{4,7\}, \cdots, \{2k-4, 2k-1\}, \{2k-2, 2k\} \bigr\}$ has relative ordered Tur\'an density $0$. 
\end{cor}

\begin{figure}[ht]
	\begin{center}
		\begin{tikzpicture}[xscale = 0.9]
			\draw 
			\foreach \i in {1,...,14}{
				(\i,0)node[vtx,label=below:$\i$](v\i){}
			};
			\draw
			\foreach  \u/\v in {1/3, 2/5, 4/7, 6/9, 8/11, 10/13, 12/14}{
				(v\u) to[bend left] (v\v)
			};
		\end{tikzpicture}
	\end{center}
	\caption{Graph $M_7$.}
	\label{fig:M_7}
\end{figure}
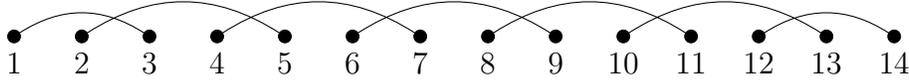

\begin{proof}
    Let $F_1 = \bigl\{ \{1,2\} \bigr\}$ and $F_{i+1} = \tilde{F_i}$ for $i \ge 1$. Then $M_k$ is a subgraph of $F_k$. From Theorem~\ref{thm:loc-ext}, we have $\rho(M_k) \le \rho(F_k) = \rho(F_1) = 0$. 
\end{proof}

\section{Conclusion}
While we found $Q_{a,b}$ satisfying \eqref{eq:question}, we did not actually determine $\rho(Q_{a,b})$.

\begin{question}
    Determine $\rho(Q_{a,b})$. 
\end{question}

Inspired by Observation~$\ref{obs:copies}$ and the specific matchings we considered, we also ask if these are merely specific cases of a broader behavior.
\begin{question}
	Is it always true that $\rho(F+_I F) = \rho(F)$, for ordered~$F$ on ~$[n]$ and~$I \in \binom{[2n]}{n}$? 
\end{question}
Finally, the results of~\cite{RRSS25} show that~$\ell(F)\geq 3 \implies \rho(F) >0$. On the other hand, the graphs with~$\ell(F) = 2$ are easy to describe. Namely, they are subgraphs of blowups of the ordered half graph~$H_t$ with~$V(H_t)=[t]$ and
    \[ E(H_t) = \bigl\{ \{i,j\} \colon i,j \in[t], \, i<j, \, i\equiv 1 \text{ and } j \equiv 0 \!\!\!\mod 2 \bigr\} \] 
for some~$t \in \NN$. Therefore showing that~$\rho(H_t) = 0$ for all~$t \in \NN$ would characterise those~$F$ with~$\rho(F) =0$ as those with~$\ell(F) = 2$. The first unknown case is the following.
\begin{question} \label{question:12-14-34}
    Determine~$\rho(H_4)$, where~$H_4$ has edges~$\{\{1,2\},\{1,4\},\{3,4\}\}$. 
\end{question}
Lior Gishboliner provided the following neat proof that~$\rho(H_4)=0$. Consider the ordered graph $G$ on $n$ vertices with $\{i,j\}\in E(G)$ whenever $|j-i|$ is a power of $3$. Then $e(G) \approx n \log n$. Suppose $G'$ is a subgraph of $G$ with at least $2n-1$ edges. For each vertex, delete the shortest edge to the left and to the right. This removes at most $2n-2$ edges, so some edge $\{x,y\}$ with $x<y$ remains. Restoring the deleted edges, there must exist $\{x,z\}$ with $x<z<y$ and $\{w,y\}$ with $x<w<y$. Since all edge lengths are powers of $3$, we have $y-x>(z-x)+(y-w)$, which implies that $\bigl\{\{x,z\},\{x,y\},\{w,y\}\bigr\}\subseteq G'$ is isomorphic to $H_4$. The next interesting case is~$H_6$.

Short of showing~$\rho(H_t) = 0$, one could ask for the more specific case of an arbitrary ordered matching~$M$, beyond those obtained by iterating Theorem~\ref{thm:loc-ext}.
\begin{question}
    Does~$\rho(M)=0$ hold for every ordered matching~$M$? 
\end{question}

\section*{Acknowledgments}
This project was started during  Graduate Research Workshop in Combinatorics 2025. The workshop was supported in part by NSF DMS-2152490, Barbara Jansons Professorship and Iowa State University. The authors thank Fares Soufan for presenting the problem during the workshop. 

\bibliographystyle{plain}
\bibliography{references}

\end{document}